\documentclass[12pt]{amsart}
\usepackage{nccmath}

\usepackage{amsmath,amsthm,amssymb}
\usepackage{times}
\usepackage{enumerate}

\newtheorem{theorem}{Theorem}[section]
\newtheorem{corollary}[theorem]{Corollary}
\newtheorem{lemma}[theorem]{Lemma}
\newtheorem{proposition}[theorem]{Proposition}

\newcommand{\be}{\begin{equation}}
\newcommand{\ee}{\end{equation}}
\newcommand{\ol}{\overline}

\newcommand{\lt}{\left}
\newcommand{\rt}{\right}

\newcommand{\goto}{\rightarrow}
\newcommand{\R}{\mathbb{R}}

\newcommand{\e}{\varepsilon}

\newcommand{\RNum}[1]{\uppercase\expandafter{\romannumeral #1\relax}}

\theoremstyle{definition}

\newtheorem{remark}[theorem]{Remark}

\numberwithin{equation}{section}

\begin{document}
\setlength{\baselineskip}{1.2\baselineskip}

\title[Complete translating solutions]
{Complete translating solitons to the mean curvature flow in $\R^3$ with nonnegative mean curvature}

\author{Joel Spruck}
\address{Department of Mathematics, Johns Hopkins University,
 Baltimore, MD 21218}
\email{js@math.jhu.edu}
\author{Ling Xiao}
\address{Department of Mathematics, University of Connecticut,
Storrs, Connecticut 06269}
\email{ling.2.xiao@uconn.edu}

\begin{abstract} We prove that any complete immersed two-sided mean convex  translating soliton $\Sigma \subset \R^3$ for the mean curvature flow is convex. As a corollary it follows that an entire mean convex graphical translating soliton  in $\R^3$ is the axisymmetric  ``bowl soliton''. We also show that if the mean curvature of $\Sigma$ tends to zero at infinity, then  $\Sigma$ can be represented as an entire graph and so is  the ``bowl soliton''. Finally we classify the asymptotic behavior of all locally strictly convex graphical translating solitons defined over strip regions.
\end{abstract}

\maketitle

\section{Introduction}
\label{sec0}

A complete immersed  hypersurface $f:\Sigma^n \goto \R^{n+1}$ with trivial normal bundle (two-sided for short) is called a translating soliton for the mean curvature flow, with respect to a unit direction $e_{n+1}$, if its
mean curvature  is given by $H=<N,e_{n+1}>$ where $N$  is a global unit normal field for $\Sigma$. Then $F(x,t):=f(x)+te_{n+1}$ satisfies
\[\Delta^{\Sigma}F=HN=<N,e_{n+1}>N=({e_{n+1}})^{\perp}=F_t^{\perp}~,\]
thus justifying the terminology.
Translating solitons  form a special class of eternal  solutions for the mean curvature flow that besides having their own intrinsic interest,  are models of slow singularity formation. Therefore there has been a great deal of effort in trying to classify them in the case $H>0$. In this paper, {\em we shall always assume our translating solitons are mean convex which by abuse of language we take to mean $H>0$}. The abundance of glueing constructions for translating solitons with high genus and $H$ changing sign (see \cite{N1},\cite{N2},\cite{N3}, \cite{Hal13}, \cite{DDN}, \cite{S})
suggests a general classification is unlikely.

For $n=1$ the unique solution is the grim reaper curve
$\Gamma:\,\,x_2= \log{\sec{x_1}}, \,\,\,|x_1| <\frac{\pi}2~,$
while for $n\geq 2\,$ we have the one parameter family of convex grim cylinders (see Lemma \ref{grim} for the $n=2$ case)
\be \label{eq0.10}
x_{n+1}=\lambda^2\log{\sec{\frac{ x_1}{\lambda}}}+ \sum_{k=2}^n \alpha_k x_k,\, \,\,\sum_{k=2}^n \alpha_k^2=\lambda^2-1,\,\,\,|x_1|<\frac{\pi}2 \lambda,\, \,\,\lambda \geq 1~,
\ee
which can be obtained from the standard grim cylinders $\Gamma \times \R^{n-1}$ by a rotation and scaling. The family of  grim graphical solitons in $\R^3$:
\be \label{0.grim}
u^{\lambda}(x_1,x_2)= \lambda^2\log{\sec{\frac{ x_1}{\lambda}}}\pm Lx_2,\, \, L=\sqrt{\lambda^2-1},\, \lambda \geq 1
\ee
defined over the strip $\mathcal{S}^{\lambda}:=\{(x_1,x_2): |x_1|<R:=\lambda \frac{\pi}2\}$ will play a central role in our classification of mean convex graphical translating solitons in $\R^3$.

Similarly if
$x_{m+1}=v(x_1,\ldots, x_m)$ is a graphical translating soliton in $\R^{m+1}$, then for $x=(x_1,\ldots,x_n)=(x',x_{m+1},\ldots,x_n)$,
\be \label{eq0.15}
v^{\lambda}(x):=\lambda^2v(\frac{ x'}{\lambda})+ \sum_{k=m+1}^n \alpha_k x_k,\, \,\,\sum_{k=m+1}^n \alpha_k^2=\lambda^2-1,\,\,\,\lambda \geq 1~,
\ee
is a graphical translating soliton in $\R^{n+1}=\R^m \times \R^{n-m}\times\R $. Conversely if $x_{n+1}=u(x_1,\ldots, x_n)$ is a {\em convex} graphical translating soliton in $\R^{n+1}$,  then $D^2 u$ has constant  rank $m,\,\,1\leq m<n \,$  by Corollary 1.3 of Bian and Guan \cite{BG09}  which implies $u(x)=v^{\lambda}(x)$ for an appropriate choice of  $x_1, \ldots, x_n$ and $\alpha_{m+1},\dots, \alpha_n$. Moreover  by the results of Wang \cite{Wang11},
$\Sigma=\text{graph}(v)$ is a complete graph in $\R^{m+1}$ defined over a strip in $\R^m$ or is an entire graph over $\R^m$.

There is as well a unique (up to horizontal translation) axisymmetric solution called the ``bowl soliton'' \cite{AW94}, \cite{CSS} which has the  asymptotic expansion as an entire graph $u(x)=\frac1{2(n-1})|x|^2 -\log{|x|}+O(1)$.

White \cite{Wh03} tentatively conjectured that all convex translating solutions have the form
\[\{(x,y,z)\in \R^j\times\R^{n-j}\times \R: z=f(|x|)\}\]
for $j\geq 2$; for $j=1,\, f$ defines the grim reaper curves so is defined on an interval. White remarks that even if this conjecture
is false, {\em it may be true for blow up limits of mean convex mean curvature flows}.
In \cite{Wang11} Wang proved that in dimension $n = 2$, any entire convex  graphical translating soliton must be rotationally symmetric, and hence the bowl soliton. He also showed there exist entire locally strictly convex graphical translating solitons for dimensions $n>2$ that are not rotationally symmetric (thus disproving one conjecture of White \cite{Wh03}) as well as complete locally strictly convex graphical translating solitons defined over strip regions in $\R^n$. Wang also conjectured that for $n=2$, any entire graphical translating soliton must be locally strictly convex and a similar statement in dimension $n>2$ under the additional assumption that $H$ tends to zero at infinity.

More recently, Haslhofer \cite{Hasl15} proved the uniqueness of the bowl soliton in arbitrary dimension under the assumption that the translating soliton $\Sigma$ is $\alpha$-noncollapsed and uniformly 2-convex.  The $\alpha$-noncollapsed condition means that for each $P\in \Sigma$, there are closed  balls $B^{\pm}$ disjoint from $\Sigma-{P}$ of radius at least $\frac{\alpha}{H(P)}$  with $ B^{+}\cap B^{-}=\{P\}$.  It figures prominently in the  regularity theory for mean convex mean curvature flow \cite{Wh00}, \cite{Wh03}, \cite{HS99}, \cite{SW09}.
The 2-convex condition (automatic if $n=2$) means that if $\kappa_n \leq \kappa_{n-1} \leq \ldots \kappa_1$ are the ordered principal curvatures of $\Sigma$, then $\kappa_n+\kappa_{n-1} \geq \beta H$ for some uniform $\beta>0$. The $\alpha$-noncollapsed condition is a deep and powerful property of weak solutions of the mean convex mean curvature flow \cite{Wh03},\cite{HK} which implies that any complete ({\em $\alpha$-noncollapsed}) mean convex translating soliton $\Sigma$ is convex with uniformly bounded second fundamental form.

The main result of this paper is a proof of a more general form of the $n=2$ conjecture of Wang.
\begin{theorem}
\label{conth.2}
Let $\Sigma\subset \R^3$ be a complete immersed two-sided translating soliton for the mean curvature flow with nonnegative mean curvature. Then $\Sigma$ is convex.
\end{theorem}

By Sacksteder's theorem \cite{S}, the condition $H>0$  and Corollary 2.1 of \cite{Wang11}, we may conclude
 \begin{corollary}
 \label{cor1}
  $\Sigma$  is the boundary of a convex region in $\R^3$ whose projection on the plane spanned by $e_1,\,e_2$ is (after rotation of coordinates)  either a strip region $\{(x_1,x_2):|x_1|< R\}$ or $\R^2$. Moreover $\Sigma$ is the graph of a function $u(x_1, x_2)$ which satisfies the equation
 \be \label{0.25}
 \text{div}\lt(\frac{\nabla u}{\sqrt{1+|\nabla u|^2}}\rt)=\frac{1}{\sqrt{1+|\nabla u|^2}}~,
 \ee
 or in nondivergence form,
 \be \label{0.26}
 (1+u_{x_2}^2)u_{x_1 x_1}-2u_{x_1}u_{x_2}u_{x_1 x_2}+(1+u_{x_1}^2)u_{x_2 x_2}=1+u_{x_1}^2+u_{x_2}^2~.
 \ee

 \end{corollary}

Combining Theorem \ref{conth.2} with Theorem 1.1 in \cite{Wang11} we have
\begin{corollary}
\label{cor2}
Any entire solution in $\R^2$ to the equation
\[\text{div}\lt(\frac{\nabla u}{\sqrt{1+|\nabla u|^2}}\rt)=\frac{1}{\sqrt{1+|\nabla u|^2}}\]
must be rotationally symmetric in an appropriate coordinate system and hence is the bowl soliton.
 \end{corollary}

A necessary and sufficient condition for mean convex  translating solitons to be graphical over $\R^2$ and thus the bowl soliton
is given in the next theorem.

\begin{theorem}
\label{th2}
Let $\Sigma\subset \R^3$ be a complete immersed two-sided translating soliton for the mean curvature flow with nonnegative mean curvature and suppose that $H(P)\goto 0$ as $P\in \Sigma$ tends to infinity. Then $\Sigma$ is after translation the axisymmetric bowl soliton.
\end{theorem}

The existence of the grim family $u^{\lambda}$ (see \ref{0.grim} ), the convexity Theorem  \ref{conth.2} and a global curvature bound (see Theorem {\ref{curvature} in section \ref{curv})
are the key tools that allows us to classify locally strictly convex graphical translating  solitons defined over strips.

\begin{theorem}\label{thm.strip} Let $\Sigma=\text{graph}(u)$  be a complete locally strictly convex translating soliton defined over a strip region $\mathcal{S}^{\lambda}:=\{(x_1,x_2): |x_1|<R:=\lambda \frac{\pi}2\}$. Then (after possibly relabeling the $e_2$ direction)\\
i. For  $\lambda\leq 1$ there is no locally strictly convex solution in $\mathcal{S}^{\lambda}$.\\
ii. $\lim_{x_2\goto +\infty} u_{x_2}(x_1,x_2)=L:=\sqrt{\lambda^2-1},\,\lambda>1$. \\
iii. $\lim_{x_2\goto -\infty} u_{x_2}(x_1,x_2)=-L$.\\
iv.
\be \label{eq.65}
\begin{aligned}
 \,&\lim_{A\goto \pm \infty}(u(x_1, x_2+A)-u(0, A))=u^{\lambda}(x_1,x_2)\\
 \,&\lim_{A\goto \pm \infty}u_{x_1}(x_1, A)=\lambda \tan{\frac{x_1}{\lambda}}~.
 \end{aligned}
 \ee
 The above limits are uniform in $|x_1| \leq R-\e \, $ for any $\e>0$.\\
 v. \,\,For $P=(x_1,x_2, u(x_1, x_2))\in \Sigma,\,\, (x_1, x_2)\in \mathcal{S}^{\lambda}$,
 \be \label{eq.66}
 H(P)\leq R-|x_1| ,\,\, H(P) \geq \theta(\e)\,\, \text{if}\,\,  |x_1| \leq R-\e~.
 \ee
 vi. $u(x_1, x_2)=u(-x_1, x_2)$  and $u_{x_1}(x_1,x_2)>0$ for $x_1>0$.\\
 \end{theorem}

\begin{remark}\label{remark1} It was shown in  \cite{Wang11} that  locally strictly convex solutions in strips do exist. The  existence of a locally strictly convex translating soliton in every $\mathcal{S}^{\lambda}$ has recently been proven by Bourni, Langford and Tinaglia in the preprint \cite{BLT18}. Moreover, the existence, uniqueness
and a complete classification of all graphical translators has just appeared in the preprint of Hoffman, Ilmanen, Martin and White \cite{HIMW18}.
\end{remark}

The organization of the paper is as follows. In section \ref{curv} we show that any immersed two-sided mean convex translating soliton is stable. This allows us to use the method of Choi-Schoen \cite{CS85} as modified by Colding-Minicozzi \cite{CM1}, \cite{CM2} to prove a global curvature bound (Theorem \ref{curvature}). This  is needed for the proof of Theorem \ref{conth.2}  in section \ref{pr}, which is based on a delicate maximum principle argument. In section \ref{pr2} we  give the proof of Theorem \ref{th2}. This result also allows us to start the  proof of Theorem \ref{thm.strip} which is long and detailed and contained  in section \ref{strip}.
We wish to thank Theodora Bourni for a careful reading and helpful comments.
\bigskip
\section{stability and curvature estimates. }
\label{curv}
We will need the the following well-known identities that hold on any translating soliton in $\R^{n+1}\,$ (see for example \cite{MSS15}).

\begin{lemma}  \label{lem0.5} Let $\Sigma$ be a two-sided immersed translating soliton in $\R^{n+1}$ with second fundamental form $A$. Let $A=(h_{ij})$ be the second fundamental form of $\Sigma,\,u=x_{n+1}\big |_\Sigma$ and $\Delta^f:=\Delta^{\Sigma}+<\nabla,e_{n+1}>$ be the drift Laplacian on $\Sigma$. Then
\begin{fleqn}
\begin{align*}
&i.  \, \,|\nabla u|^2 = 1-H^2,\, \Delta^{\Sigma} u=H^2\\
&ii.  \,\,\Delta^f A+|A|^2 A=0,\\
&iii. \, \, \Delta^f H+H|A|^2 =0,\\
&iv.\,\,  \Delta^f (|A|)^2 -2|\nabla A|^2 + 2|A|^4 = 0.
\end{align*}
\end{fleqn}
\end{lemma}

It is well known that a translating soliton $\Sigma$ with respect to the direction $e_{n+1}$ in $\R^{n+1}$ is a critical
point of the weighted area functional
\[ \tilde{\mathcal{A}} (\Sigma) =\int_{\Sigma}e^{x_{n+1}}dv\]
 and is in fact minimal with repect to the weighted Euclidean metric $e^{x_{n+1}}\delta$ on $\R^{n+1}$.
The second variation of $\tilde{\mathcal{A}}$ with respect to a compactly supported normal variation $\eta N$ is easily computed to be
\[ \tilde{\mathcal{A}}''(0)=\int_{\Sigma} (|\nabla \eta|^2-|A|^2 \eta^2)e^{x_{n+1}}dv=\int_{\Sigma}-\eta L\eta \,e^{x_{n+1}}\,dv~,\]
where $L\eta=e^{-x_{n+1}}\text{div}^{\Sigma}(e^{x_{n+1}}\nabla \eta)+|A|^2 \eta\,$ is the associated stability operator.
\begin{proposition}\label{prop1} i. Let $\Sigma\subset \R^{n+1}$ be a complete immersed two-sided translating soliton with respect to $e_{n+1}$ with $H\geq 0$. Then $\Sigma$ is strictly stable, that is, $\lambda_1(-L)(D)>0$ on any compact subdomain $D \subset \Sigma$.\\
ii. The following are equivalent:\\
a. There exists a positive solution $v$ of $Lv=0$ for every bounded D.\\
b. $\lambda_1(-L)(D)\geq 0$ for every bounded D.\\
c.  $\lambda_1(-L)(D)>0$ for  every bounded D .
\end{proposition}
\begin{proof}
i. We may assume that $H>0$. Then $w=\log{H}$ satisfies $ \Delta w+<\nabla w, e_{n+1}>+|A|^2=-|\nabla w|^2$.
Hence
\begin{align*}
\int_{\Sigma} (\eta^2 |A|^2 -|\nabla \eta |^2) e^{x_{n+1}}dv&=-\int_{\Sigma}(\eta^2 |\nabla w|^2-2\eta <\nabla \eta, \nabla w>+|\nabla \eta|^2)e^{x_{n+1}}dv\\
&=-\int_{\Sigma} |\eta \nabla w-\nabla \eta |^2 e^{x_{n+1}}dv <0.
\end{align*}
ii. The proof is a straightforward modification of that of  Fischer-Colbrie and Schoen \cite{FS} and will not be given.
\end{proof}

We will need the following corollary for the case at hand $n=2$.
\begin{corollary} \label{cor1.1}
Let $\mathcal{B}_{\rho}(P)$ be an intrinsic ball  in $ \Sigma $ with $\rho<2\pi $. Then $ \mathcal{B}_{\rho}(P)$
is disjoint from the conjugate locus of P and
\be \label{eq10}
 \int_{\Sigma} f^2 |A|^2 dv \leq e^{2\rho} \int_{\Sigma} |\nabla f|^2 ~dv
 \ee
for $f \in H^1_0(\mathcal{B}_{\rho}(P))$.
\end{corollary}
\begin{proof} Let $K=K^{\Sigma}$ denote the Gauss curvature of $\Sigma$. Since $K\leq \frac{H^2}4\leq \frac14$, the first statement follows from standard comparison geometry.
Also $|\nabla x_3|^2=1-H^2 \leq 1$, so $|x_3(Q)-x_3(P)|\leq \rho $ for any $Q\in \mathcal{B}_{\rho}(P)$.
Hence $     e^{-\rho} e^{x_3}(P) \leq  e^{x_3}(Q) \leq e^{\rho} e^{x_3}(P) $
 and the result follows from Proposition \ref{prop1}.
 \end{proof}

We now follow the  Colding-Minicozzi method \cite{CM1, CM2} with appropriate modification, to prove  intrinsic area bounds and then curvature bounds. For two dimensional graphs, such curvature estimates follow immediately from the work of Leon Simon \cite{S77}, see also \cite{Sh}.
 \begin{proposition} \label{prop2} Let $\Sigma \subset \R^3$ be a two-sided immersed translating soliton with $H\geq 0$ and let $\mathcal{B}_{\rho}(P)$ be a topological disk in $\Sigma$. Then $\mathcal{B}_{\rho}(P)$ is disjoint from the cut locus of P for $e^{2\rho}<2$ and
\be \label{eq20}
 \begin{aligned}
i. \,\, & \frac{\text{area}(\mathcal{B}_{\rho}(P))}{\rho^2} \leq  2\pi~,\\
ii.\,\, &\int_{\mathcal{B}_{\mu^2 \rho}(P)}|A|^2 ~dv \leq 2\pi \{(\log{\frac1{\mu}})^{-2}+2(\log{\frac1{\mu}})^{-1}\}\,\, \text{for  $\mu\in (0,1)$.}
 \end{aligned}
 \ee
 \end{proposition}
\begin{proof} We first prove $\mathcal{B}_{\rho}(P)$ is disjoint from the cut locus $C(P)$ of P, that is the injectivity radius of P  satisfies $r_0:=\text{inj}_{p}(\Sigma)>\rho$. Suppose  for contradiction that $Q\in \partial \mathcal{B}_{r_0}(P)$ is a cut locus of P and $r_0\leq \rho$ . We know by Corollary \ref{cor1} that Q is not in the conjugate locus of P so by Klingenberg's lemma (see for example Lemma 5.7.12 of \cite{P})  there are two mininimizing geodesics from P to Q which fit together smoothly at Q with a possible corner at P and  bounding a domain $D \subset \mathcal{B}_{r_0}(P)$. By Gauss-Bonnet,
\[ \frac14 \text{area}(D)\geq \int_{D}K~dv=2\pi-\int_{\partial D}\kappa_g~d\sigma \geq \pi~,\]
hence $\text{area}(\mathcal{B}_{r_0}(P))>\text{area}(D)\geq 4\pi$. On the other hand by \eqref{eq20} with $\rho=r_0$  (proved below), $\text{area}(\mathcal{B}_{r_0}(P))<2\pi r_0^2$. Therefore $r_0>\sqrt{2}$, contradicting $r_0 \leq \rho<\frac12 \log{2}$.

We now prove the stated inequalities. Let $l(s)$ be the length of $\partial \mathcal{B}_{s}(P)$ and $K(s)=\int_{\mathcal{B}_{s}(P)}K~dv$. By Gauss-Bonnet,
\be \label{eq30}
l'(s)=\int_{\partial \mathcal{B}_{s}(P)}\kappa_g~d\sigma=2\pi-K(s)~.
\ee
 For $r=d(P,x),\, f(x)=\eta(r),\, \eta,-\eta'\geq 0, \eta(\rho)=0$, we use the stability inequality \eqref{eq10}
and write
\[ |A|^2=H^2-2K \geq -2K,\]
to obtain (for $\rho\leq r_0$)
\be \label{eq40}
-2\int_{\mathcal{B}_{\rho}(P)}Kf^2 \leq 2\int_{\mathcal{B}_{\rho}(P)}|\nabla f|^2.\\
\ee
By the coarea formula,
\be \label{eq50}
\begin{aligned}
&\int_{\mathcal{B}_{\rho}(P)}Kf^2=\int_0^{\rho}\eta^2(s)\int_{\partial \mathcal{B}_{s}(P)}Kd\sigma ds=\int_0^{\rho}\eta^2(s)K'(s)ds,\\
&\int_{\mathcal{B}_{\rho}(P)}|\nabla f|^2=\int_0^{\rho}\int_{\partial \mathcal{B}_{s}(P)}|\nabla f|^2d\sigma ds=\int_0^{\rho}(\eta')^2(s)l(s)ds.
\end{aligned}
\ee

For $\eta(r)=1-\frac{r}{\rho}$, using \eqref{eq30}-\eqref{eq50} this gives
\be \label{eq60}
-\frac{4}{\rho}\int_0^{\rho}(2\pi-l'(s))(1-\frac{s}{\rho})ds \leq 2\frac{\text{area}(\mathcal{B}_{\rho}(P))}{\rho^2}.
\ee
By integration by parts,
\be \label{eq70}
\int_0^{\rho}(2\pi-l'(s))(1-\frac{s}{\rho})ds =\pi \rho-\frac{\text{area}(\mathcal{B}_{\rho}(P))}{\rho}.
\ee
Finally combining \eqref{eq60}, \eqref{eq70} and simplifying gives  inequality i. in \eqref{eq20}.
For part ii. we use the logarithmic cutoff
\be \label{eq80} \eta(s)=\left \{
\begin{aligned}&\hspace{.2in}1 \hspace{.5in} &\text{if \,\,$s\leq \mu^2 \rho$}\\
&\frac{\log{\frac{s}{\rho}}} {\log{\mu}}-1 &\text{if\,\, $ \mu^2 \rho<s<\mu \rho$}\\
&\hspace{.2in}0  &\text{if \,\,$s>\mu \rho$}
\end{aligned}\right.
\ee
in \eqref{eq10}. Then
\be \label{eq90}
\begin{aligned}
&\int_{\mathcal{B}_{\mu^2\rho}(P)}|A|^2 \leq \frac2{(\log{\mu})^2}\int_{\mu^2 \rho}^{\mu \rho}\frac{l(s)}{s^2}ds\\
&\leq  \frac2{(\log{\mu})^2}\big\{\frac{ \text{area}(\mathcal{B}_{s}(P))}{s^2}\Bigg|_{\mu^2 \rho}^{\mu \rho}+2\int_{\mu^2 \rho}^{\mu \rho}\frac{\text{area}(\mathcal{B}_{s}(P))}{s^3}ds\big\}\\
&\leq 4\pi \{\frac1{(\log{\mu})^2}+\frac2{\log{\frac1{\mu}}}\}
\end{aligned}
\ee
by \eqref{eq20} part i.
\end{proof}

For later use we will need two well known lemmas; the first is an extrinsic mean value inequality (for a proof see \cite{CM2} p. 26-27) and the second one says that curvature bounds implies graphical with intrinsic balls
$\mathcal{B}_s$ and extrinsic balls $B_s$ related (see \cite{CM2} Lemma 2.4).

\begin{lemma}\label{meanvalue} Let $ \Sigma \subset \R^3$ be an embedded surface with $x_0\in \Sigma,\, \,B_s(x_0)\cap \partial \Sigma=\emptyset$. Suppose the mean curvature of $\Sigma$ satisfies
$|H| \leq C$ and $f\geq 0$ is a smooth function on $\Sigma$ satisfying $\Delta^{\Sigma} f\geq -\lambda s^{-2} f$.
Then
\[ f(x_0) \leq \frac{e^{( \frac{\lambda}4+Cs)}}{\pi s^2} \int_{B_s(x_0)\cap \Sigma} f dv ~.\]
\end{lemma}

\begin{lemma}\label{graph} Let $ \Sigma \subset \R^3$ be an immersed surface with $16s^2 \sup_{\Sigma} |A|^2 \leq 1$. If $P\in \Sigma$ and $\text{dist}^{\Sigma}(P,\partial \Sigma)\geq 2s$, then\\
i. $\mathcal{B}_{2s}(P)$ can be written as a graph of a function u over $T_{P}\Sigma$ with $|\nabla u|\leq 1$ and $ \sqrt{2}s |\text{Hess}_u| \leq 1$;\\
ii. The connected component of $B_s(P)\cap \Sigma$ containing P is contained in $\mathcal{B}_{2s}(P)$.
\end{lemma}

\begin{proposition}(Choi-Schoen type curvature bound) \label{prop3}  Let $\Sigma \subset \R^3$ be a two-sided immersed translating soliton and let $\mathcal{B}_{\rho}(P)$ be  disjoint from the cut locus of P.  Then there exists $\e,\, \tau<\frac{\sqrt{\e}}{2\pi}<\rho$ such that if for all $r_0\leq \tau$, there holds $\int_{\mathcal{B}_{r_0}(P)}|A|^2\leq  \e$, then for all $ \,0<\sigma \leq r_0,\,\, y \in \mathcal{B}_{r_0-\sigma}(P)$ we have $|A|^2(y) \leq  \sigma^{-2}$.
\end{proposition}

\begin{proof} Define $F(x)=(r_0-r(x))^2 |A|^2(x)$ on $\mathcal{B}_{r_0}(P)$ and suppose F assumes its maximum at
$x_0$.  Note that if $F(x_0)\leq 1$, then $\sigma^2 |A|^2(y)\leq (r_0-r(y))^2|A|^2(y) \leq F(x_0) \leq 1$ and we are done.
If not, define $\sigma$ by $4\sigma^2|A|^2(x_0)=1$. Then by the triangle inequality on $\mathcal{B}_{\sigma}(x_0)$,
\[\frac12 \leq \frac{r_0-r(x)}{r_0-r(x_0)} \leq 2 ~,\]
which implies $\sup_{\mathcal{B}_{\sigma}(x_0)} |A|^2 \leq 4|A|^2(x_0)= \sigma^{-2}$. On $\Sigma$ we have the Simons'
type equation $L(|A|^2)-2|\nabla A|^2+2|A|^4=0$ which implies $\Delta (|A|^2+\frac12) \geq -2|A|^2(|A|^2+\frac12)$. Hence for $f(x):=|A|^2+\frac12,\, \Delta f \geq - 2\sigma^{-2} f \,\,\,\text{on $\mathcal{B}_{\sigma}(x_0)$}$. Using Lemmas \ref{graph} and \ref{meanvalue} with $s=\frac{\sigma}4,\, \lambda=\frac16,\, C=1$ and Proposition \ref{prop2} part i., we find
\be \label{eq110}
(\frac1{4\sigma^2}+\frac12)=|A|^2(x_0)+\frac12 \leq \frac{16}{\pi \sigma^2}e^{(\frac1{24}+\frac{\sigma}4)}\{\e +2\pi r_0^2\} ~.
\ee
Multiplying \eqref{eq110} by $\sigma^2$ we find
\be \label{eq120}
\frac14<\frac14+\frac{\sigma^2}2 \leq \frac{16}{\pi}e^{(\frac1{24}+\frac{\sigma}4)}\{\e +2\pi r_0^2\} <\frac{32e}{\pi}\e
\ee
which is impossible for $\e\leq \frac{\pi}{128e}$.
\end{proof}

We are now in a position to prove the curvature estimate we will need in the next section.

\begin{theorem} \label{curvature}Let $\Sigma\subset \R^{3}$ be a complete immersed two-sided translating soliton with respect to $e_{3}$ with $H\geq 0$. Then there is a universal constant C such that $|A|^2(P)\leq C$ for $P\in \Sigma$.
\end{theorem}
\begin{proof}  For any $P\in \Sigma$, we fix $\rho>0$ such that $e^{\rho}<2$ as in Proposition \ref{prop2} so that $\mathcal{B}_{\rho}(P)$ is  disjoint from the cut locus of P. We may assume $\mathcal{B}_{\rho}(P)$ is a topological disk since by Proposition \ref{prop1}, the universal cover of $\mathcal{B}_{\rho}(P)$ endowed with pull-back metric is also a stable translating soliton with nonnegative mean curvature.
Thus using Proposition \ref{prop2} part ii. with $\mu=e^{-\frac{6\pi}{\e}}$, the conditions of Proposition \ref{prop3} are satisfied for $\tau=\mu^2 \rho$. We can choose $\sigma=\tau$ and obtain $|A|^2(P) \leq \tau^{-2}$.
\end{proof}

\section{Proof of Theorem \ref{conth.2}.}
\label{pr}
We restate for the readers convenience our main result.
\begin{theorem}
Let $\Sigma\subset \R^3$ be a complete immersed two-sided translating soliton for the mean curvature flow with nonnegative mean curvature. Then $\Sigma$ is convex.
\end{theorem}
\begin{proof}
Without loss of generality, we assume $\Sigma$ satisfies
\[H=\lt<N, e_{3}\rt>>0.\]
Let \be\label{con.8}
f(x_1, x_2)=\frac{x_1+x_2}{2}+\lt[\lt(\frac{x_1-x_2}{2}\rt)^2\rt]^{1/2},
\ee
\be\label{con.9}
\phi(r)=\left\{\begin{aligned}
&r^4e^{-1/r^2}\,\, &\mbox{if $r<0$}\\
&0\,\, &\mbox{if $r\geq 0$,}
\end{aligned}\right.
\ee
and
\be\label{con.10}
g(z)=f(z)\sum\limits_{i=1}^2\phi\lt(\frac{z_i}{f(z)}\rt).
\ee
It's easy to see that $f$ and $g$ are smooth when $z_1\neq z_2.$  Now denote
\[G(A)=g(\kappa(A))\,\,\mbox{and $F(A)=f(\kappa(A)),$}\]
where $A$ is a $2\times 2$ symmetric matrix and $\kappa(A)$ are the eigenvalues of $A.$
Now let $A=(h_{ij})$ be the second fundamental form of $\Sigma$.  If we order the principle curvatures $\kappa_1 \geq \kappa_2$ of $\Sigma$,  then $ \frac{G}{F}=\phi(\frac{\kappa_2}{\kappa_1}) \geq 0$ is smooth when $\kappa_1>0>\kappa_2$. Since  $G/F\leq1,\,G/F$ achieves its
maximum either at an interior point or ``at infinity''.
 \begin{lemma} On $\Sigma\cap\{\kappa_1>0> \kappa_2\}$,
\be \label{con.11}
 \begin{aligned}
&\Delta^{\Sigma}\lt(\frac{G}{F}\rt)+2\lt<\frac{\nabla F}{F}, \nabla\lt(\frac{G}{F}\rt)\rt>+\lt<\nabla\lt(\frac{G}{F}\rt), e_{3}\rt>\\
&=-\frac{G^{ij}h_{ij}|A|^2}{F}+\frac{GF^{ij}h_{ij}|A|^2}{F^2}+\lt(\frac{G^{pq, rs}}{F}-\frac{GF^{pq, rs}}{F^2}\rt)h_{pqk}h_{rsk},\\
\end{aligned}
\ee
\end{lemma}
where $G^{ij}=\frac{\partial G}{\partial a_{ij}},\, F^{ij}=\frac{\partial F}{\partial a_{ij}},\, G^{pq,rs}=\frac{\partial^2 G}{\partial a_{pq} \partial a_{rs}},\, F^{pq,rs}=\frac{\partial^2 F}{\partial a_{pq} \partial a_{rs}}$.
\begin{proof} Let $\tau_1, \tau_2$ be a local orthonormal frame on $\Sigma$. Then using Lemma \ref{lem0.5}
\be \label{con.115}
\begin{aligned}
&\nabla_k (\frac{G}{F})=(\frac{G}{F})^{pq}h_{pqk}=(\frac{FG^{pq}-GF^{pq}}{F^2})h_{pqk}\\
&\Delta^{\Sigma}(\frac{G}{F})=(\frac{G}{F})^{pq}\Delta h_{pq}+(\frac{G}{F})^{pq,rs}h_{pqk}h_{rsk}\\
&=-|A|^2(\frac{FG^{ij}h_{ij}-GF^{ij}h_{ij}}{F^2})-<\nabla(\frac{G}{F}), e_3>+(\frac{G}{F})^{pq,rs}h_{pqk}h_{rsk}~.
\end{aligned}
\ee
But \be\label{con.116}
(\frac{G}{F})^{pq,rs}h_{pqk}h_{rsk}=(\frac{G^{pq, rs}}{F}-\frac{GF^{pq, rs}}{F^2})h_{pqk}h_{rsk}-2<\nabla(\frac{G}{F}), \frac{\nabla F}{F}>,
\ee
and together \eqref{con.115}, \eqref{con.116} give \eqref{con.11}.
\end{proof}

We next compute the last term in \eqref{con.11}. We use the notation $g^p=\frac{\partial g}{\partial \kappa_p},\,
g^{pq}=\frac{\partial^2 g}{\partial \kappa_p \partial \kappa_q}$ and similarly for f. As is now well-known (see for example \cite{A1994})
\be\label{con.12}
G^{pq, rs}h_{pqk}h_{rsk}=g^{pq}h_{ppk}h_{qqk}+2\frac{g^2-g^1}{\kappa_2-\kappa_1}h^2_{12k},
\ee
where
\be\label{con.13}
\begin{aligned}
g^{pq}&=f^{pq}\sum\limits_{i=1}^2\lt[\phi\lt(\frac{z_i}{f}\rt)-\frac{z_i}{f}\dot{\phi}\lt(\frac{z_i}{f}\rt)\rt]\\
&+\frac{1}{f}\sum\limits_{i=1}^2\ddot{\phi}\lt(\frac{z_i}{f}\rt)\lt(\delta_i^p-\frac{z_i}{f}f^p\rt)\lt(\delta_i^q-\frac{z_i}{f}f^q\rt).\\
\end{aligned}
\ee
It follows that
\be\label{con.14}
\begin{aligned}
&FG^{pq, rs}-GF^{pq, rs}\\
&=\lt(fg^{pq}-gf^{pq}\rt)+2\lt\{f\frac{g^2-g^1}{\kappa_2-\kappa_1}-g\frac{f^2-f^1}{\kappa_2-\kappa_1}\rt\}\\
&=I+II.
\end{aligned}
\ee
We proceed to calculate the terms I and  II.
\be\label{con.15}
\begin{aligned}
I&=f^{pq}\lt[g-\sum\limits_{i=1}^2z_i\dot{\phi}\lt(\frac{z_i}{f}\rt)\rt]
+\sum\limits_{i=1}^2\ddot{\phi}\lt(\frac{z_i}{f}\rt)\lt(\delta_i^p-\frac{z_i}{f}f^p\rt)\lt(\delta_i^q-\frac{z_i}{f}f^q\rt)-gf^{pq}\\
&=-f^{pq}\sum\limits_{i=1}^2z_i\dot{\phi}\lt(\frac{z_i}{f}\rt)+\sum\limits_{i=1}^2\ddot{\phi}\lt(\frac{z_i}{f}\rt)
\lt(\delta_i^p-\frac{z_i}{f}f^p\rt)\lt(\delta_i^q-\frac{z_i}{f}f^q\rt),
\end{aligned}
\ee
\be\label{con.16}
\begin{aligned}
g^2-g^1&=\dot{\phi}\lt(\frac{z_2}{f}\rt)+f^2\sum\limits_{i=1}^2\lt[\phi\lt(\frac{z_i}{f}\rt)-\frac{z_i}{f}\dot{\phi}\lt(\frac{z_i}{f}\rt)\rt]\\
&-\dot{\phi}\lt(\frac{z_1}{f}\rt)-f^1\sum\limits_{i=1}^2\lt[\phi\lt(\frac{z_i}{f}\rt)-\frac{z_i}{f}\phi\lt(\frac{z_i}{f}\rt)\rt]\\
\end{aligned}
\ee
and
\be\label{con.17}
\begin{aligned}
&\frac{\kappa_2-\kappa_1}2 II=f(g^2-g^1)-g(f^2-f^1)\\
&=f\lt[\dot{\phi}\lt(\frac{z_2}{f}\rt)-\dot{\phi}\lt(\frac{z_1}{f}\rt)\rt]+f^2\lt[g-\sum\limits_{i=1}^2z_i\dot{\phi}\lt(\frac{z_i}{f}\rt)\rt]\\
&-f^1\lt[g-\sum\limits_{i=1}^2z_i\dot{\phi}\lt(\frac{z_i}{f}\rt)\rt]-g\lt(f^2-f^1\rt)\\
&=f\lt[\dot{\phi}\lt(\frac{z_2}{f}\rt)-\dot{\phi}\lt(\frac{z_1}{f}\rt)\rt]+\sum\limits_{i=1}^2z_i\dot{\phi}\lt(\frac{z_i}{f}\rt)(f^1-f^2).\\
\end{aligned}
\ee

Assume $\kappa_1>0>\kappa_2$; then
\be\label{con.18}
\begin{aligned}
&\Delta^{\Sigma}\lt(\frac{G}{F}\rt)+2\lt<\frac{\nabla F}{F}, \nabla\lt(\frac{G}{F}\rt)\rt>+\lt<\nabla\lt(\frac{G}{F}\rt), e_{3}\rt>\\
&=(\frac{G^{pq, rs}}{F}-\frac{GF^{pq, rs}}{F^2})h_{pqk}h_{rsk}
=\frac{1}{F^2}(FG^{pq, rs}-GF^{pq, rs})h_{pqk}h_{rsk}\\
&=\frac{1}{\kappa_1^2}\ddot{\phi}\lt(\frac{\kappa_2}{\kappa_1}\rt)\lt(\delta_2^p
-\frac{\kappa_2}{\kappa_1}f^p\rt)\lt(\delta_2^q-\frac{\kappa_2}{\kappa_1}f^q\rt)h_{ppk}h_{qqk}
+\frac{2}{\kappa_1^2}\dot{\phi}\lt(\frac{\kappa_2}{\kappa_1}\rt)\lt(\frac{\kappa_1+\kappa_2}{\kappa_2-\kappa_1}\rt)h^2_{12k}\geq 0.
\end{aligned}
\ee
Here we used $f^{pq}=0,$ $\dot{\phi}\lt(\frac{\kappa_1}{f}\rt)=0,$ and $f^2=0$ when $\kappa_1>0>\kappa_2.$
It follows that if $G/F$ achieves its maximum at an interior point, then by the strong maximum principle we have $\frac{G}{F}\equiv \mbox{constant}.$
If $\kappa_2/\kappa_1\neq 0$, then $\varphi:=\log{\frac{|A|^2}{H^2}}$ is constant and so
\be\label{con.1*}
\begin{aligned}
 0&=\nabla \varphi=\frac{\nabla |A|^2}{|A|^2}-2\frac{\nabla H}{H}\\
 0&=\Delta^{\Sigma}  \varphi+\lt<\nabla \varphi, e_{3}\rt>\\
 &=2\frac{|\nabla A|^2}{|A|^2}-\frac{|\nabla |A|^2|^2}{|A|^4}+2\frac{|\nabla H|^2}{H^2}\\
 &=\frac2{|A|^2}( |\nabla A|^2 -|\nabla |A| |^2) ~.
\end{aligned}
\ee
Hence $|\nabla A|^2=|\nabla |A||^2$ and so $h_{12k}=0,\, k=1,2$. By the Codazzi
equations, $h_{112}=h_{221}=0$. Since $h_{22}=r_0 h_{11}$, we deduce $\nabla A=0$.
Thus $M$ is a complete surface with constant mean curvature. Since $M$ satisfies
$H=\lt<N, e_{3}\rt>$,
we conclude that $M$ is a plane,  a contradiction.
Therefore in this case  we must have $\frac{G}{F}\equiv 0,$ and thus $\kappa_2\geq 0.$

If $G/F$ achieves its maximum at infinity, by Theorem \ref{curvature} we may apply the Omori-Yau maximum principle and conclude that
there exists a sequence ${P_n}$ tending to infinity with $r_n:=\kappa_2/\kappa_1(P_n)\goto r_0,$ where $-1\leq r_0<0.$  Introduce a local orthonormal frame $\tau_1, \tau_2$ in a neighborhood of $P_n$ which diagonalizes $(h_{ij}(P_n))$. Then we have at ${P_n}$:
\be\label{fix.2}
\begin{aligned}
\frac{1}{n}&\geq\ddot{\phi}(r_n)\lt[\frac{r_nh_{11k}}{\kappa_1}-\frac{h_{22k}}{\kappa_1}\rt]^2
+\frac{2}{\kappa_1^2}\dot{\phi}(r_n)\frac{1+r_n}{r_n-1}h_{12k}^2-2\lt<\frac{\nabla F}{F},\nabla\lt(\frac{G}{F}\rt)\rt>\\
&=\ddot{\phi}(r_n)\lt[\frac{r_nh_{11k}}{\kappa_1}-\frac{h_{22k}}{\kappa_1}\rt]^2
+\frac{2}{\kappa_1^2}\dot{\phi}(r_n)\frac{1+r_n}{r_n-1}h_{12k}^2-2\dot{\phi}(r_n)\lt<\frac{h_{11k}}{h_{11}}, \frac{h_{22k}}{h_{11}}-r_n\frac{h_{11k}}{h_{11}}\rt>
\end{aligned}
\ee
and
\be \label{fix.25}
C_{n, k}:=\frac{h_{22k}}{h_{11}}-r_n\frac{h_{11k}}{h_{11}}\goto 0\,\,\mbox{as $n\goto\infty$}.
\ee
Note that
\be\label{fix.3*}
\begin{aligned}
&\tilde{C}_{n,k}:=\frac{\nabla_k H}{\kappa_1}=\frac{-\kappa_k\lt<\tau_k, e_{3}\rt>}{\kappa_1}\\
&=\frac{h_{22k}}{h_{11}}-r_n\frac{h_{11k}}{h_{11}}+(1+r_n)\frac{h_{11k}}{h_{11}}=C_{n,k}+(1+r_n)\frac{h_{11k}}{h_{11}}
\end{aligned}
\ee
From \eqref{fix.25} we have,
\be\label{eq.1}
\frac{h_{22k}}{\kappa_1}=-|r_n|\frac{h_{11k}}{\kappa_1}+C_{n, k},\,\,k=1,2
\ee
and
\be\label{eq.2}
\frac{h_{11k}h_{22k}}{\kappa_1^2}=-|r_n|\lt(\frac{h_{11k}}{\kappa_1}\rt)^2+C_{n, k}\frac{h_{11k}}{\kappa_1}.
\ee
\text{Claim:} $(1+r_n)\lt|\frac{h_{11k}}{h_{11}}\rt|\goto 0,\,\, k=1, 2.$\\
If not, we can choose a subsequence, still denoted by $\{P_n\},$ so that for $n\geq N_0$
\be\label{eq.3}
(1+r_n)\lt|\frac{h_{112}}{\kappa_1}\rt|(P_n)\geq \e_0,\,\,(1+r_n)\lt|\frac{h_{221}}{\kappa_1}\rt|(P_n)\geq \e_0.
\ee
Then from \eqref{fix.2} we have
\be\label{eq.4}
\begin{aligned}
\frac{1}{n}&\geq\ddot{\phi}\sum C_{n, k}^2-2\dot{\phi}\sum C_{n, k}\frac{h_{11k}}{\kappa_1}
-\frac{2\dot{\phi}}{1+|r_n|}(1+r_n)\lt(\frac{h_{112}^2}{\kappa_1^2}+\frac{h_{221}^2}{\kappa_1^2}\rt)\\
&=\ddot{\phi}\sum C_{n, k}^2-2\dot{\phi}\lt(C_{n,1}\lt(-\frac{1}{|r_n|}\frac{h_{221}}{\kappa_1}+\frac{C_{n, 1}}{|r_n|}\rt)\right.\\
&\left.+C_{n, 2}\frac{h_{112}}{\kappa_1}+\frac{1+r_n}{1+|r_n|}\lt(\frac{h_{112}^2}{\kappa_1^2}+\frac{h_{221}^2}{\kappa_1^2}\rt)\rt)\\
&=\ddot{\phi}\sum C_{n, k}^2-2\dot{\phi}\frac{C_{n, 1}^2}{|r_n|}\\
&-2\dot{\phi}\lt\{\lt(\frac{1+r_n}{1+|r_n|}\frac{h_{112}^2}{\kappa_1^2}+C_{n, 2}\frac{h_{112}}{\kappa_1}\rt)
+\lt(\frac{1+r_n}{1+|r_n|}\frac{h_{221}^2}{\kappa_1^2}-\frac{C_{n, 1}}{|r_n|}\frac{h_{221}}{\kappa_1}\rt)\rt\}\\
&\geq-\frac{2\dot{\phi}}{1+|r_n|}\lt(\lt|\frac{h_{112}}{\kappa_1}\rt|\lt((1+r_n)\lt|\frac{h_{112}}{\kappa_1}\rt|-C_{n, 2}(1+|r_n|)\rt)\right.\\
&\left.+\lt|\frac{h_{221}}{\kappa_1}\rt|\lt((1+r_n)\lt|\frac{h_{221}}{\kappa_1}\rt|-\frac{C_{n, 1}}{|r_n|}(1+|r_n|)\rt)\rt)\\
&\geq-2\frac{\dot{\phi}}{1+|r_n|}\frac{\e_0}{2}\lt(\lt|\frac{h_{112}}{\kappa_1}\rt|+\lt|\frac{h_{221}}{\kappa_1}\rt|\rt),\\
\end{aligned}
\ee
which leads to a contradiction for $n\geq N_0$ by \eqref{eq.3}. Thus the claim is proven and so $\tilde{C}_{n,k} \goto 0$.
Therefore  $N(P_n)$ converges to $e_3$ and so $H (P_n)\goto 1.$

Now let $\Sigma_n=\Sigma-P_n$ be the surface obtained from $\Sigma$ by translation of $P_n$ to the origin. Since $\Sigma$ has bounded principle curvatures, so do the  $\Sigma_n.$ Choosing a subsequence which we still denote by $\Sigma_n,$ the
$\Sigma_n$ converge smoothly to $\Sigma_\infty.$ Thus $\Sigma_\infty$ is a translating soliton which satisfies $H=\lt<N, e_{3}\rt>,$ and $H(0)=1.$ Moreover, we have
\[\inf\limits_{x\in \Sigma_\infty}\frac{\kappa_2}{\kappa_1}=\frac{\kappa_2}{\kappa_1}(0).\]
As before we conclude that $G/F= \text{constant},$ and $\Sigma_\infty$ has constant mean curvature one, which is impossible.
Therefore $\Sigma$ is convex.
\end{proof}

\bigskip

\section{Proof of Theorem \ref{th2}.}
\label{pr2}

In this section we give  the proof of
\begin{theorem}
Let $\Sigma\subset \R^3$ be a complete immersed two-sided translating soliton for the mean curvature flow with positive mean curvature and suppose that $H(P)\goto 0$ as $P\in \Sigma$ tends to infinity. Then $\Sigma$ is after translation the axisymmetric bowl soliton.
\end{theorem}

\begin{proof} By Corollaries \ref{cor1}, \ref{cor2} we may
suppose for contradiction that $\Sigma=\text{graph(u)}$ projects onto the strip  $|x_1|<R$.  Consider for any $A$, the  convex curve $\gamma^A(x_1)=(x_1, A, u(x_1,A): -R\leq x_1 \leq R)$. Since $u(x_1,A) \goto +\infty$ as $x_1 \goto \pm R$, there is a smallest value $x_1=x_1(A)$ so that
$u_{x_1}(x_1(A),A)=0$.  Since $\Sigma $ is not a grim cylinder we may assume by a suitable  change of coordinates that
$u_{x_2}(x_1(0), 0)\geq \delta>0$. We  normalize $u(x_1(0), 0)=0$ ; then by convexity of $u$,
\be \label{eq.10}
\begin{aligned}
\,\, &u(x_1(A),A) \geq u(x_1(0),0)+Au_{x_2}(x_1(0), 0)\geq A\delta,\,  \\
&0=u(x_1(0),0)\geq u(x_1(A),A)-Au_{x_2}(x_1(A),A) \geq A(\delta-u_{x_2}(x_1(A),A)),
\end{aligned}
\ee
which implies $u_{x_2}(x_1(A),A)\geq \delta$. By the assumption that $H(P)\goto 0$ as $P\in \Sigma$ tends to infinity,
we have that $|D u(x_1(A),A)|=|u_{x_2}(x_1(A),A)|= u_{x_2}(x_1(A),A)\goto \infty$ as $A\goto \infty$.
Therefore
\be \label{eq.15}
u(x_1,x_2) \geq  u(x_1(A),A)+(x_2-A)u_{x_2}(x_1(A),A)\geq A\delta+(x_2-A)u_{x_2}(x_1(A),A),
\ee
and so choosing $x_2=B,\, A=\frac{B}2$,
\be \label{eq.20}
\lim_{B\goto \infty} \frac{u(x_1,B)}{B} \geq \frac{\delta}2 +\frac12 \lim_{B\goto \infty}u_{x_2}(x_1(\frac{B}2),\frac{B}2)=+\infty~.
\ee
Hence $u$ grows superlinearly as $x_2 \goto \infty$.\\

We now compare $\Sigma$ with a tilted cylinder of radius R.  Consider the parametrized family of graphs
\[x_3=v^{t}(x_1,x_2):=-\sqrt{1+{t}^2}\sqrt{R^2-x_1^2}+t(x_2-A),\, |x_1| \leq R,\,x_2\geq A\]
of constant mean curvature $H=\frac1{R}$ with respect to upward normal direction. Since  $ v^{t}(x_1,x_2)\leq t(x_2-A)$, for any choice of $t\geq 0,\, u>v^{t}$ for $x_2$ sufficiently large by \eqref{eq.20}.  Also, $u>v^{t}$
  for $|x_1|=R$ and for $x_2=A$. For $t \leq \delta,\,u>v^{t}$ in $x_2 \geq A,\,|x_1|\leq R$ by \eqref{eq.15}. Note that
  $\lim_{t \goto \infty}v^{t}(x_1,3A)\geq \lim_{t \goto \infty}(2t A-\sqrt{1+t^2}\,R)=+\infty \,\,\text{for $A>R$}$.
  We can therefore increase $t$ until there is a first contact of $\Sigma=\text{graph($u$)}$ and $\text{graph}(v^{t})$,  which must occur over an interior point of the half-strip $\{(x_1, x_2),\, |x_1|<R,\, x_2>A\}$. This gives a
$P:=(x_1,x_2, u(x_1,x_2)) \in \Sigma,\, x_2 >A$ with $H(P) \geq \frac1{R}$. Since $A$ is arbitrary we have a contradiction.
\end{proof}

\section{asymptotic behavior of complete locally strictly convex translating solitons }
\label{strip}
 In this section we study the asymptotic behavior of complete locally strictly convex translating solitons.

\begin{lemma}\label{grim}Let $\Sigma=\text{graph}(u)$ be a complete mean convex translating soliton in $\R^3$ and suppose that the smallest principle curvature $\kappa_2$ vanishes at a point of $\Sigma$. Then $\kappa_2\equiv 0$ everywhere and
after translation, $\Sigma$ is grim cylinder of the form $\Sigma=\text{graph}(u^{\lambda})$ defined in a strip $\{(x_1,x_2): |x_1|<R\}$ where
\[u^{\lambda}(x_1,x_2):=\lambda^2\log{\sec{(\frac{ x_1}{\lambda})}}\pm\sqrt{\lambda^2-1}\,x_2,\,\, R=\frac{\pi}2 \lambda,\,\lambda\geq 1~.\]
In particular if $\Sigma$ contains a line, then $\Sigma$ is a grim cylinder of the above form.
\end{lemma}
\begin{proof}
If we choose an orthonormal frame $\tau_1,\tau_2$ so that $\kappa_2(P)=h_{22}(P)$, then $\kappa_2\equiv 0<\kappa_1 $ on $\Sigma$ by Lemma \ref{lem0.5} part ii. and the maximum principle. Thus the Gauss curvature $K^{\Sigma}\equiv 0$ and so
$\Sigma$ has a representation (see for example \cite{HN59}) $\Sigma=\text{graph}(z)$ where $z(x_1,x_2)=\eta(x_1)+\alpha x_2$ for a constant $\alpha$ and a scalar function $\eta$ defined in a simply connected region containing the projection of $P$ on the $x_1, x_2$ plane. Therefore
\be \label{eq.30}
(1+\alpha^2)\eta''=1+\alpha^2+\eta'^2.
\ee
Set $\tilde{\eta}(x_1)=\lambda^{-2} \eta(\lambda x_1)$, where $\lambda^2=1+\alpha^2$. Then
\be \label{eq.40}
\tilde{\eta}''=1+\tilde{\eta}'^2
\ee
and so (up to translation of coordinates and an additive constant)
\be \label{eq.45}
\tilde{\eta}(x_1)=\log{\sec{x_1}}~,
\ee
 which proves the lemma.
\end{proof}

We have the following Harnack inequality for the mean curvature H on $\Sigma$ (compare Hamilton \cite{Ham95}, Corollary 1.2).
\begin{lemma}\label{harnack} For any two points $P_1,\,P_2 \in \Sigma$,
\be \label{eq.60} H(P_2)  \geq e^{-d^{\Sigma}(P_1,P_2)}H(P_1)~.
\ee
\end{lemma}
\begin{proof} Since $H=<N,e_3>,\, \nabla_k H=-\kappa_k<\tau_k,e_3>$ so that
\[|\nabla H|^2 \leq |A|^2=H^2-2K\leq H^2~.\]
Therefore, $|\nabla \log{H}|\leq 1$ and \eqref{eq.60} follows.
\end{proof}

\begin{lemma}\label{grad} Let $\Sigma=\text{graph}(u)$ be a complete mean convex translating soliton in $\R^3$ defined in the strip $\{(x_1,x_2): |x_1|<R\}$. Then
 \[H(x_1,x_2):=H((x_1,x_2,u(x_1,x_2)) \leq R-|x_1|~.\]
\end{lemma}
\begin{proof} We use that $\frac{\sum u_{ij}^2}{W^6} \leq |A|^2\leq H^2=\frac1{W^2}$ where $W^2=1+ |\nabla u|^2$.
Then $|DW|\leq W^2$ or $|DH|=|D(\frac1W)|\leq 1$. By the convexity of u and the fact that $u(x_1, x_2)\goto \infty$ as
$R-|x_1| \goto 0,\, \text{$x_2$\, fixed}$,  we see that $H(x_1, x_2)\goto 0$ as $R-|x_1| \goto 0,\, \text{$x_2$\, fixed}$. Hence
$H(x_1,x_2) \leq \min{(R-x_1, x_1+R)}= R-|x_1|$.
\end{proof}

\begin{lemma}\label{arctan}  Let $\Sigma=\text{graph}(u)$ be a  convex translating soliton in $\R^3$ defined in the strip $\{(x_1,x_2): |x_1|<R\}$. Then
\be \label{eq69}
\begin{aligned}
\, &|\frac{d}{dx_i}\arctan{u_{x_i}}| \leq 1,\, \,\, i=1,2,\\
\, &|\frac{d}{dx_2}\sqrt{1+(u_{x_1})^2}\ |\leq |u_{x_1}|\sqrt{1+(u_{x_2})^2}.
\end{aligned}
\ee
\end{lemma}

\begin{proof}
Since $u$  is a convex graphical solution to the translating soliton equation,
 \be \label{eq71}
 \begin{aligned}
 \,&(1+u_{x_2}^2)u_{x_1 x_1}+(1+u_{x_1}^2)u_{x_2 x_2}=2u_{x_1}u_{x_2}u_{x_1 x_2}+(1+u_{x_1}^2+u_{x_2}^2)\\
 \,&\leq 2|u_{x_1}| |u_{x_2}| \sqrt{u_{x_1 x_1}u_{x_2 x_2}}+(1+u_{x_1}^2+u_{x_2}^2)\\
 \,&\leq (1+u_{x_1}^2)u_{x_2 x_2}+\frac{u_{x_1}^2 u_{x_2}^2}{(1+u_{x_1}^2)}u_{x_1 x_1}+(1+u_{x_1}^2+u_{x_2}^2)~.
 \end{aligned}
 \ee
 This implies
 \be \label{eq72}
 \begin{aligned}
  \,&u_{x_1 x_1}\leq 1+u_{x_1}^2 \, \, \text {and by symmetry $\,\, u_{x_2 x_2}\leq 1+u_{x_2}^2 $},\\
  \,&| u_{x_1 x_2}| \leq \sqrt{u_{x_1 x_1}u_{x_2 x_2}}\leq \sqrt{1+u_{x_1}^2}\sqrt{1+u_{x_2}^2}~,
  \end{aligned}
  \ee
  and \eqref{eq69} follows.
\end{proof}

\begin{lemma}\label{entire}
Let $\Sigma=\text{graph}(u)$ be a complete mean convex translating soliton in $\R^3$ defined over $\R^2$. Then $H(P) \goto 0$ for $P\in \Sigma$ tending to infinity.
\end{lemma}
\begin{proof} We slightly modify the argument of Haslhofer \cite{Hasl15}. Fix $P_0 \in \Sigma$ and suppose there is a sequence $P_n\in \Sigma$ tending to infinity with $\liminf_{n\goto \infty} H(P_n) >0$. Passing to a subsequence, we may assume
$\frac{P_n-P_0}{|P_n-P_0|}$ converges to a unit direction $\omega$. Let $\Sigma_n=\Sigma-P_n$ be the surface obtained from $\Sigma$ by translation of $P_n$ to the origin. Since $\Sigma$ has bounded principle curvatures, so do the  $\Sigma_n.$ Choosing a  subsequence which we still denote by $\Sigma_n,$ the
$\Sigma_n$ converge smoothly to $\Sigma_\infty$, a convex complete translating soliton.  Moreover, $\Sigma_\infty$ is not a vertical plane since $H(0)>0$, so must be a graph. Since the region K above $\Sigma$ is convex and $\frac{P_n-P_0}{|P_n-P_0|}\goto \omega$, the limit $\Sigma_\infty$ contains a line (see Lemma 3.1 of \cite{BL16} for more details). Therefore by Lemma \ref{grim}, $\Sigma_\infty$ is a  grim cylinder and thus is a graph over a strip, a contradiction.
\end{proof}

\begin{lemma} \label{pr2.cor}  Let $\Sigma=\text{graph}(u)$  be a complete locally strictly convex translating soliton defined over a strip region $\mathcal{S}^{\lambda}:=\{(x_1,x_2): |x_1|<R:=\lambda \frac{\pi}2\}$. Then there exist sequences $P_n=(x_1^n, x_2^n, u(x_1^n, x_2^n)),\, \ol{P}_n=(\ol{x}_1^n, \ol{x}_2^n, u(\ol{x}_1^n, \ol{x}_2^n))\in \Sigma$ with $x_2^n \goto \infty,\, \ol{x}_2^n\goto -\infty$ and $H(P_n), H(\ol{P}_n) \geq \theta>0$.
\end{lemma}

\begin{proof} By Theorem \ref{th2} there is a sequence of points $P_n=(x_1^n, x_2^n, u(x_1^n, x_2^n))\in \Sigma$ where (after relabeling axes if necessary) $\liminf_{n\goto \infty}H(P_n)\geq \theta>0$ and $x_2^n \goto +\infty$. Note that by Lemma \ref{grad}, $ |x_1^n| \leq R-\theta$.
We claim there is also a sequence $\ol{P}_n=(\ol{x}_1^n, \ol{x}_2^n, u(\ol{x}_1^n, \ol{x}_2^n))\in \Sigma$ with $\ol{x}_2^n\goto -\infty$ and $ H(\ol{P}_n) \geq \theta>0$. If not, then for $x_2\goto -\infty,\, H(x_1,x_2):= H(x_1,x_2, u(x_1,x_2)) \goto 0$.
For any $B\leq 0$ let $x_1(B)$ be the unique value so that $u_{x_1}(x_1(B),B)=0$. Since $\Sigma$ is not a grim cylinder,
we may assume by a translation of coordinates that $u_{x_2}(x_1(0),0) = \delta>0$. We normalize $u(x_1(0), 0)=0$; then
\[u(x_1(B), B) \geq u(x_1(0),0)+Bu_{x_2}(x_1(0),0)=B\delta~,\]
and
\[ 0=u(x_1(0), 0)\geq u(x_1(B), B)-Bu_{x_2}(x_1(B), B)\geq B(\delta-u_{x_2}(x_1(B), B))~.\]
 Hence $u_{x_2}(x_1(B), B))\leq \delta$. Since $ H(x_1,x_2) \goto 0$ as $x_2\goto -\infty$, we conclude
 \be \label{eq60.5}
 u_{x_2}(x_1(B), B)) \goto -\infty \,\,\text{as $ B \goto -\infty$}.
 \ee
 Therefore
  \be \label{eq.61}
  u(x_1,x_2) \geq u(x_1(B),B) +(x_2-B)u_{x_2}(x_1(B), B) \geq B\delta+ (x_2-B)u_{x_2}(x_1(B), B)
  \ee
Now choose $x_2=\Lambda<0$ and $B=\frac{\Lambda}2$. Then
\[u(x_1, \Lambda)\geq \frac{\delta \Lambda}2+ \frac{\Lambda}2 u_{x_2}(x_1(\frac{\Lambda}2), \frac {\Lambda}2)~,\]
hence
\be \label{eq.62}
\lim_{\Lambda \goto -\infty} \frac{u(x_1, \Lambda)}{\Lambda}\leq \frac{\delta}2+ \frac12 u_{x_2}(x_1(\frac{\Lambda}2), \frac{\Lambda}2)\goto -\infty
\ee
by \eqref{eq60.5}. Thus $u(x_1, \Lambda) \goto \infty$ superlinearly as $\Lambda \goto -\infty$. We now compare $\Sigma$ with a tilted cylinder of radius R. Consider $x_3=v^t (x_1,x_2):=-\sqrt{1+t^2} \sqrt{R^2-x_1^2}+t(x_2-B)$ in the half-strip
$\mathcal{S}^B:=\{(x_1,x_2): |x_1|\leq R,\, x_2 \leq B<0\}$. Since $v^t \leq t(x_2-B)$, for any choice of $t\leq 0,\, u>v^t$ for $x_2$ sufficiently small (i.e. $x_2$ large negative) by \eqref{eq.62}. Also $u>v^t $ for $|x_1|=R$ and $x_2=B<0$, as soon as $u(x_1, B)>0$.
For $t \geq -\delta, \, u>v^t$ in $\mathcal{S}^B$ and since
\[ \lim_{t\goto -\infty}v^t(x_1, 3B) \geq \lim_{t\goto -\infty}(2Bt-\sqrt{1+t^2}\,R) \goto \infty~,\]
we can decrease t until there is a first contact point $P \in \mathcal{S}^B$ of $\text{graph}(v^t)$ and $\Sigma$. At $P\in \Sigma,\,
H(P) \geq \frac1R,\, $ a contradiction.

\end{proof}

\begin{lemma} \label{maxprinc} Let $\Sigma=\text{graph}(u)$  be a complete locally strictly convex translating soliton defined over a strip  $\mathcal{S}_R:=\{(x_1,x_2): |x_1|<R\}$ with $\max_{ |x_1|\leq R-\delta}W(x_1,0)\leq C_0$. If $W:=\sqrt{1+|\nabla u|^2}$ satisfies $\sup_{x_2\geq 0}W(0,x_2)\leq C_1$, then
\[W(x_1,x_2)\leq  2\frac{R-\delta}{\delta}\max{(C_0,C_1)}\,\,\text{in the half-strip $\mathcal{S}^{+}_{R-\frac32 \delta}:=\mathcal{S}_{R-\frac32 \delta}\cap \{x_2 \geq 0\}$}~.\]
\end{lemma}
\begin{proof}   We give the proof for the right half  of $\mathcal{S}^{+}_{R-\frac32 \delta}$ (where $0\leq x_1 \leq
R-\frac32 \delta$); the proof for the left half of $\mathcal{S}_{R-\frac32 \delta}$ is analogous.

For an n-dimensional hypersurface $\Sigma=\text{graph}(u)$,
\[ \Delta^f H:=\Delta^{\Sigma}H+<\nabla^{\Sigma}H, e_{n+1}>=a^{ij}H_{x_i x_j}~,\]
 where $a^{ij}:=\delta_{ij}-\frac{u_i u_j}{W^2}$.
 Since $H=\frac1W$, we obtain from Lemma \ref{lem0.5} part iii.,
\be \label{eq64}
\mathcal{L}W:= (a^{ij}\partial_i \partial_j-\frac2{W}a^{ij}W_i \partial_j)W=\frac{|A|^2}{W^3} \geq \frac1{2W}
\ee
for $n=2$.
For $N>1$ fixed and large, set $\eta=\eta^N:=(1-\frac{x_1}{R-\delta}-\frac{x_2}N)_{+}$ in the domain
$D=\{0<x_1<R-\delta,\, x_2>0\}$. Since $\eta$ is linear in $D\cap\{\eta>0\}$ and
\[\mathcal{L}(\eta W)=\eta \mathcal{L}W+Wa^{ij}\eta_{ij}\geq \frac{\eta}{2W}~,\]
$\eta W$ cannot have an interior maximum in $D\cap\{\eta>0\}$.   Thus $\eta W$ achieves its maximum when $x_1=0$ or $x_2=0$. Restricting to $D':=\{0\leq x_1 \leq R-\frac{3 \delta}2,\, x_2 \geq 0\}$ and letting $N\goto \infty$ gives \[\frac{\delta}{2(R-\delta)}W \leq \eta W \leq \max{(C_0, C_1)}\,\,\text{ in $D'$}~.\]
This completes the proof.
\end{proof}

We now analyze the asymptotic behavior of complete locally strictly convex translating solitons $\Sigma=\text{graph}(u)$ in $\R^3$ that are defined over the strip region $\mathcal{S}^{\lambda}$.

\begin{theorem}\label{thm.strip1} Let $\Sigma=\text{graph}(u)$  be a complete locally strictly convex translating soliton defined over a strip region $\mathcal{S}^{\lambda}:=\{(x_1,x_2): |x_1|<R:=\lambda \frac{\pi}2,\, \lambda\geq 1\}$. Then (after possibly relabeling the $e_2$ direction)\\
i. For  $\lambda\leq 1$ there is no locally strictly convex solution in $\mathcal{S}^{\lambda}$.\\
ii. $\lim_{x_2\goto +\infty} u_{x_2}(x_1,x_2)=L:=\sqrt{\lambda^2-1},\,\lambda>1$.\\
iii. $\lim_{x_2\goto -\infty} u_{x_2}(x_1,x_2)=-L$.\\
iv.
\be \label{eq.65}
\begin{aligned}
 \,&\lim_{A\goto \pm \infty}(u(x_1, x_2+A)-u(0, A))=u^{\lambda}(x_1,x_2)\\
 \,&\lim_{A\goto \pm \infty}u_{x_1}(x_1, A)=\lambda \tan{\frac{x_1}{\lambda}}~.
 \end{aligned}
 \ee
 The above limits hold uniformly on $\mathcal{S}^{\e}=\{(x_1, x_2)\in \mathcal{S}^{\lambda}: |x_1| \leq R-\e\}$.\\
 v. \,\,For $P=(x_1,x_2, u(x_1, x_2))\in \Sigma,\,\, (x_1, x_2)\in \mathcal{S}^{\e}$,
 \be \label{eq.66}
 H(P)\geq \theta(\e) ~.
 \ee
 \end{theorem}

\begin{proof}

By Lemma \ref{pr2.cor}  there is a sequence of points $P_n \in \Sigma$ where (after relabeling axes if necessary) $\liminf_{n\goto \infty}H(P_n)\geq \theta>0$ and $x_2^n \goto +\infty$. Note that by Lemma \ref{grad},
$ |x_1^n| \leq R-\theta$. Arguing as in Lemma \ref{entire},
$\Sigma_n:=\Sigma-P_n$ converge smoothly to $\Sigma_\infty$, a complete convex translating soliton defined over a strip $\mathcal{S}^{\lambda^{\prime}}-(x_1^{\infty},0)$ passing through the origin.  Note that to begin with we can only assert that $\lambda^{\prime}\leq \lambda$.
By Lemma \ref{grim}, $\lambda^{\prime}>1$ and $u(y_1+x_1^n, y_2+x_2^n)-u(x_1^n, x_2^n)$ converges locally smoothly to
\[ (\lambda^{\prime})^2\log{\sec{(\frac{ y_1+x_1^{\infty}}{\lambda^{\prime}})}}+\sqrt{(\lambda^{\prime})^2-1}\,y_2-
(\lambda^{\prime})^2 \log{\sec{\frac{x_1^{\infty}}{\lambda^{\prime}}}}\]
 as $n\goto \infty$ for $|y_1+x_1^{\infty}|<R':=\frac{\pi}2 \lambda'$. In particular,
 \be \label{eq76}
 \begin{aligned}
 \ &  i. \lim_{n\goto \infty}u_{x_1}(y_1+x_1^n, y_2+x_2^n) = \lambda' \tan{\frac{y_1+x_1^{\infty}}{\lambda'}}~,\\
 \ &  ii. \lim_{n\goto \infty}u_{x_2}(y_1+x_1^n, y_2+x_2^n) = L':=\sqrt{(\lambda')^2-1}~.
 \end{aligned}
 \ee

 Using Lemma \ref{arctan} we see that
 \be \label{eq76.5}
 \begin{aligned}
 & |\arctan{u_{x_1}(y_1+o(1)+x_1^{\infty},y_2+x_2^n)}-\arctan{u_{x_1}(y_1+x_1^n,y_2+x_2^n)}|\\
 &\leq |x_1^n-x_1^{\infty}+o(1)| \goto 0 \,\,\text{as $n\goto \infty$},\,\,\,\text{hence by \eqref{eq76} i.}\\
 & \lim_{n\goto \infty}u_{x_1}(y_1+o(1)+x_1^{\infty}, y_2+x_2^n) = \lambda' \tan{\frac{y_1+x_1^{\infty}}{\lambda'}}.
 \end{aligned}
 \ee

Similarly using Lemma \ref{arctan} and \eqref{eq76.5},
 \be \label{eq76.7}
 \begin{aligned}
 \ & |\arctan{u_{x_2}(y_1+x_1^{\infty},y_2+x_2^n)}-\arctan{u_{x_2}(y_1+x_1^n,y_2+x_2^n)}|\\
\ &  \leq \frac{\sqrt{1+u_{x_1}^2}}{\sqrt{1+u_{x_2}^2}}(y_1+o(1)+x_1^{\infty},y_2+x_2^n)|x_1^n-x_1^{\infty}|
\goto 0 \,\,\text{as $n\goto \infty$}~,\\
\ & \text{(so that by \eqref{eq76} ii.,) \ $\lim_{n\goto \infty}u_{x_2}(y_1+x_1^n, y_2+x_2^n) = L':=\sqrt{(\lambda')^2-1}$}~.
\end{aligned}
\ee
Setting $y_1=x_1-x_1^{\infty},\,y_2=x_2$, it follows that
\be \label{eq77}
\begin{aligned}
\ & i. \lim_{n\goto \infty}(u(x_1, x_2+x_2^n)-u(x_1^{\infty}, x_2^n))= u^{\lambda'}(x_1, x_2)-(\lambda^{\prime})^2 \log{\sec{\frac{x_1^{\infty}}{\lambda^{\prime}}}}\\
\ & ii. \lim_{n\goto \infty}u_{x_1}(x_1,x_2+x_2^n)=\lambda' \tan{\frac{x_1}{\lambda'}}\\
\ & iii. \lim_{n \goto \infty} u_{x_2} (x_1, x_2+x_2^n)=L'
\end{aligned}
\ee
for $|x_1|<\lambda'$.

From \eqref{eq77} iii. we conclude that
  $ \limsup_{x_2\goto \infty}u_{x_2}(x_1, x_2)\geq L'$. But
 $u_{x_2}(x_1, x_2)\leq u_{x_2}(x_1, x_2+x_2^n)$ for n large so
 \be \label{eq79}
 \lim_{x_2 \goto \infty}u_{x_2}(x_1, x_2)=L' \leq L,\, |x_1|<\lambda'~.
\ee
In exactly the same way, we may prove
\be \label{eq80}
 \lim_{x_2\goto -\infty}u_{x_2}(x_1, x_2)=L''\leq L,\, |x_1|<\lambda''~.
\ee

Note that by \eqref{eq79},\,\eqref{eq80} for $x_2=A,\, x_1=x_1(A)$ (recall $u_{x_1}(x_1(A),A)=0$), \\
$H((x_1(A),A, u(x_1(A),A))\geq \frac1{\sqrt{1+L^2}}$. Hence we may choose $x_2^n \goto \infty$ arbitrary and $x_1^n=x_1(x_2^n)$. From \eqref{eq76} i. with $y_1=y_2=0$, we conclude
  $x_1^{\infty}=0$. Since the choice of $x_2^n \goto \infty$ is arbitrary, it follows from \eqref{eq77} ii.
  (with $x_1=x_2=0$) that

  \be \label{eq81}
  \lim_{x_2\goto \infty} u_{x_1}(0, x_2)=0~,
  \ee
and therefore (by \eqref{eq79})

\be \label{eq82}
\lim_{x_2\goto \infty} W(0,x_2)=\lambda'~,
\ee
where $W:=\sqrt{1+|Du|^2}$.\\

We can now prove that there is no drop of width in the strip of convergence, i.e $\lambda'= \lambda$ (and similarly for $\lambda''=\lambda$). Suppose for contradiction that $R'=R-2\delta<R$ and
set $C_0:=\sup_{|x_1|<R-\delta}W(x_1,0)$. By \eqref{eq82}, $C_1:=\sup_{x_2\geq 0}W(0,x_2)<\infty$.
Applying Lemma \ref{maxprinc}, we conclude
$W(x_1,x_2)\leq  2\frac{R-\delta}{\delta}\max{(C_0,C_1)}$ in the half-strip $\{(x_1,x_2): |x_1| \leq R-\frac32 \delta,\,x_2 \geq 0\}$, contradicting the completeness of $\Sigma_{\infty}$. This completes the proof of
parts i.ii., iii. and part iv. follows immediately from \eqref{eq77} i., ii. since $x_1^n=x_1(x_2^n),\, x_1^{\infty}=0$
and $x_2^n$ is arbitrary.

 To prove part v. we will use Lemma \ref{harnack} with
 \[P_2=(x_1,x_2, u(x_1,x_2)),\, P_1=(x_1(x_2), x_2, u(x_1(x_2),x_2)),\,H(P_1)\geq \frac1{\sqrt{1+L^2}}~.\]
 Normalize $u$ by $|\nabla u(x_1(0), 0)|=0$.
 We observe that
 \[d^{\Sigma}(P_2, P_1) \leq L(x_2):=\int_{-R+\e}^{R-\e}\sqrt{1+u_{x_1}^2(t,x_2)}~dt\]
 and by \eqref{eq.65}, $L(x_2) \leq M(\e)$  for $|x_2| \geq A(\e)$ sufficiently large.  Therefore using Lemma \ref{arctan}, $|L^{\prime}(x_2)|\leq \sqrt{1+L^2}\,L(x_2)$ and so
\be \label{eq.74}
d^{\Sigma}(P_2, P_1)\leq L(x_2)\leq e^{\sqrt{1+L^2}(A(\e)-|x_2|)}M(\e) \leq  e^{\sqrt{1+L^2}A(\e)}M(\e)=:\ol{M}(\e)
\ee
for $|x_2| \leq A(\e)$.
Therefore $H(P_1) \geq \theta(\e):=\frac{e^{-\ol{M}(\e)}}{\sqrt{1+L^2}} $, completing the proof.
\end{proof}

An immediate application of Theorem \ref{thm.strip1} is the following symmetry result.
 \begin{theorem} \label{thm.strip2} Let $\Sigma=\text{graph}(u)$  be a complete locally strictly convex translating soliton defined over a strip region $\mathcal{S}^{\lambda}$. Then $u(x_1,x_2)=u(-x_1,x_2)$ and $u_{x_1}(x_1,x_2)>0$ for $x_1>0$.
 \end{theorem}
 \begin{proof} We employ the method of moving planes; see for example \cite{BCN96}. Set $ \mathcal{S}_t=\{(x_1,x_2)\in \mathcal{S}^{\lambda}: t<x_1<R\}$ for $ t\in (0,R)$.
 We want to show that the function
 \be \label{eq.80}
 v^t(x_1,x_2):=u(x_1,x_2)-u(2t-x_1,x_2)>0 \,\,\text{in $\mathcal{S}_t$ for all $t\in (0,R)$}~.
 \ee
 Once \eqref{eq.80} is proven, the conclusion of Theorem \ref{thm.strip2} follows easily. Indeed letting $t\goto 0$
 in \eqref{eq.80} implies by continuity
 \be \label{eq.90}
 u(x_1,x_2)\geq u(-x_1,x_2)\,\, \text{for $0\leq x_1 \leq R$}~.
 \ee
 Since we may replace $x_1$ by $-x_1$ in \eqref{eq.90}, we have equality and thus we have symmetry in $x_1$. From \eqref{eq.80} we may also conclude that $u_{x_1}\geq 0$ for $0< x_1<R$. Since $u_{x_1}$ satisfies an elliptic equation
 without zeroth order term, we have from the maximum principle either $u_{x_1}>0$ or $u_{x_1}\equiv 0$. Since the latter possibility is impossible, Theorem \ref{thm.strip2} follows.\\
  Set $u^t=u(2t-x_1,x_2)$. Since  $ u$ and $u^t$ both satisfy the same elliptic equation \eqref{0.25} in $\mathcal{S}_t$, the difference $v^t=u-u^t$ satisfies a linear elliptic equation
  \be \label{eq.100} \sum_{i,j=1}^2 A^{ij}v^t_{ij}+\sum_{i=1}^2 B^k v^t_k=0 \,\, \text{in $\mathcal{S}_t$}~,
 \ee
 and so $v^t$ cannot have an interior minimum in $\mathcal{S}_t$. Also $v^t=0$ for $x_1=t,\,\,\lim_{x_1\goto R}v^t=+\infty$, and by Theorem \ref{thm.strip1} part iv.,
 \be \label{eq.110}
 \lim_{x_2 \goto \pm \infty} v^t =\lambda^2 \log{\{\frac{\sec{\frac{x_1}{\lambda}}}{\sec{\frac{(2t-x_1)}{\lambda}}}\}}\geq 0.
 \ee
Therefore $v_t\geq 0$ and so by the maximum principle \eqref{eq.80} holds completing the proof.
 \end{proof}

\bigskip

\end{document}